\documentclass{amsart}
\usepackage{amsmath}
\usepackage{amsthm} 
\usepackage{epsfig}
\usepackage[usenames]{color} 
\usepackage{ulem} 
\usepackage{verbatim}
\newtheorem{theorem}{Theorem}[section]
\newtheorem{conjecture}{Conjecture}
\newtheorem{lemma}{Lemma}[section]
\newtheorem{definition}{Definition}
\newtheorem{corollary}{Corollary}
\newtheorem{remark}{Remark}

\def\haef{Haefliger CR structure}

\def\bR{\mathbf{R}}
\def\bC{\mathbf{C}}
\def\cnk{\mathbf{C}^{n+k}}
\def\scriptO{\mathcal{O}}
\def\scriptM{\mathcal{M}}
\def\scriptB{{\mathcal B}}
\def\scriptR{{\mathcal R}}
\def\scriptF{{\mathcal{F}}}
\def\jone{J^1(M,X )}
\def\Jone{J^1(M,X )}
\def\iff{if and only if }
\def\ctm{{\bf{C}}\otimes T(M)}
\def\ctx{{\bf C}\otimes T(X)}

\def\tonezero{T^{1,0}}
\def\Tonezero{T^{1,0}}
\def\Tzeroone{T^{0,1}}
\def\tzeroone{T^{0,1}}
\def\CN{C^N\ }
\def\dpartial{\bar{\partial}}
\def\nbd{neighborhood\ }

\DeclareMathOperator{\rank}{rank}
\begin{document}

\numberwithin{equation}{section}

\title{CR structures on open manifolds}

\author {Howard Jacobowitz}
\address {Department of Mathematical Sciences, Rutgers University, Camden, New Jersey, USA}
\email {jacobowi@camden.rutgers.edu}

\author {Peter Landweber}
\address{Department of Mathematics, Rutgers University, Piscataway, NJ 08854, USA}
\email{landwebe@math.rutgers.edu}


\date{May 7, 2014}
\begin{abstract}
We show that the vanishing of the higher dimensional homology groups of a manifold ensures that every almost CR structure of codimension $k$ may be homotoped to a CR structure.  This result is proved by adapting a method due to Haefliger used to study foliations (and previously applied to study the relation between almost complex and complex structures on manifolds) to the case of (almost) CR structures on open manifolds.

\end{abstract}
\maketitle

\def\neighborhood{neighborhood}
\def\iff{if and only if }
\def\comtanM{C\scriptOtimes T(M)\ }
\def\T10{T^{1,0}C^N}
\def\CN{C^N\ }
\def\dpartial{\bar{\partial}}
\def\bea{\begin{eqnarray*}}
\def\eea{\end{eqnarray*}}

\medskip
\section{Introduction}
An almost complex structure on an even dimensional smooth manifold $M$
is a smooth linear vector bundle map $J$ on the tangent bundle of $M$
satisfying $J^2 = -I$. We pass to the complexified tangent bundle and
denote by $B$ the bundle of eigenspaces of $J$  with eigenvalue $+i$.  
So the elements of $B$
correspond to the anti-holomorphic vectors.  We have
\[
B \cap \overline  B = \{0\} \mbox{  and  }
B \oplus \overline B = \ctm .
\]

These steps are reversible, and in this paper we shall view an almost complex structure as being such a subbundle of the complexified tangent bundle.   


Being given an almost complex structure $B$ on $M$, one says that $B$ defines a complex structure on $M$ if in an open \nbd of each point complex coordinates may be introduced
so that $B$ is spanned by the set of vectors
\[
\frac \partial {\partial\overline{z}_j},\quad j=1,\ldots , \frac 1 2 \dim M.
\]
The Newlander-Nirenberg Theorem \cite{NN} asserts that $B$ defines a
complex structure if and only if $B$ is involutive, by which we mean
that the bracket of two smooth sections of $B$ defined on an open
subset of $M$ is also a section of $B$, i.e., the space $\scriptB$ of 
smooth vector fields with values in $B$ satisfies 
\[
[\scriptB, \scriptB] \subset \scriptB.
\]
More generally, any subbundle of $\ctm$ is said to be involutive if
this condition is satisfied.

This, of course, is reminiscent of the Frobenius condition for
foliations.  It is to be expected that various techniques developed in
the theory of foliations would have analogues whenever subbundles of
$\ctm$ are involutive.
In particular, we ask: When does $M$ admit an involutive subbundle $B$ which satisfies 
$B \cap \overline B = \{0\}$ and has a specified complex dimension?
For almost complex structures these techniques do extend, see \cite{Ad}, \cite{Land}.  Using somewhat different
techniques from
foliation theory, we provide an answer for almost CR structures on manifolds having vanishing homology in high dimensions.  

There are two basic steps to our argument.  Following Haefliger's work in foliation theory, we broaden the definition of an almost CR structure 
to a more flexible category (in Section \ref{Haefliger}) in order to apply (in Section \ref{lift}) homotopy theory to study the existence  of liftings of a continuous map $X \to B$ to a map $X \to E$, where $E \to B$ is a given fibration.
This is used to construct a manifold $X$, foliated by
complex manifolds, and a map $f: \ctm \to \tonezero (X)$ with
appropriate properties.  We then use Gromov's h-principle to prove an
analogue of the Gromov-Phillips Theorem in order to replace $f$ by a map
$F: M \to X$ whose differential has these properties (specifically,
Conditions 1 and 2 of Theorem \ref{gromov}) and thereby induces on $M$
the desired CR structure.

For convenience, here is the main result. Definitions are in the
next section. 

\begin{theorem} \label{mainresult}
If
\[
H_{p}(M^{2n+k};\, \mathbf Z)=0 \,\text{ for }\, p\geq n+k+1
\]
then every smooth   almost CR structure of codimension $k$ on $M$ is
homotopic to a $C^\omega$ CR structure of codimension $k$.  
In particular, every $C^\infty$ CR structure may be deformed to a
$C^\omega$ CR structure.
\end{theorem}


\section{The basics}
Let $M$ be a manifold of dimension $2n+k$ with $k>0$.  Let $N$
stand for some unspecified manifold of the same dimension, which will
often be just an open ball in $\bR ^{2n+k}$.  All manifolds, bundles,
and maps are of class $C^\infty$, and all manifolds are paracompact, unless otherwise
indicated. 
(It is likely that the results discussed here also hold
with only minimal smoothness assumptions.)  
All open sets which are
introduced to state local results are taken to be connected and
sufficiently small.

\begin{description}
\item
	[An almost CR structure of codimension $k$] on $M$ is a complex subbundle
    $B\subset \ctm$ of complex dimension $n$ that satisfies $B\cap {\overline B}=\{ 0 \} .$
\item 	
	[A CR structure of codimension $k$] is an almost  CR structure $B$ of codimension $k$ that in addition is involutive. 
\item 
	[A generic CR  immersion] is an immersion $f$  of $M$ into a complex
    manifold $X$ of complex dimension $n+k$ such that $(\ctm)\cap f^*\Tzeroone (X )$
    has complex dimension $n$ at all points of $M$.  We set $B=(\ctm)\cap
    f^*\Tzeroone (X)$ and observe that such an immersion induces a CR
    structure on $M$.  Conversely, when the CR structure $B$ is given
    we require that the immersion induces this CR structure.
\end{description}

\begin{remark}
Results similar to ours should hold for a large class of involutive systems.  
\end{remark}

Two things are  well known:
\begin{itemize}
\item
Not all CR structures can be obtained by
immersions into complex manifolds. That is, there exist smooth CR
structures that  are not induced by local
embeddings into $\bC ^{n+k}$. Nirenberg \cite[page 13]{Ni} gave the 
first examples of such non-realizable CR structures.

\item
Real analytic
CR structures are obtained by immersions into $\bC ^{n+k}$ and there
is even an associated uniqueness result.  For the convenience of the
reader we now state and prove this result.
\end{itemize}

\begin{lemma}\label{CR}
If $M$ has a $C^\omega$ CR structure $B$ of codimension  k then there exists an open covering 
\[
M=\bigcup _j \scriptO _j
\]
and $C^\omega $ generic CR  embeddings
\[
f_j:\scriptO _j\to \bC ^{n+k} .\]
Further, for each pair $(i,j)$ with  $\scriptO _i \cap \scriptO _j\neq \emptyset $ there exists an open set $U _{ij}$ containing $f_i(\scriptO _i\cap \scriptO _j )$ and a biholomorphism $\gamma _{ij}:U_{ji}\to U_{ij}$ with
\[
f_i=\gamma _{ij}\circ f_j \text{ on } \scriptO _i\cap \scriptO_j.
\]
\end{lemma}

\begin{proof} 
We reduce this lemma to the corresponding result for $C^\omega$ almost
complex manifolds by applying the following theorem.

\begin{theorem}[\cite{EF}, \cite{Li}]
Let $Q_1, \ldots ,Q_N$ be real analytic vector fields in a \nbd of
the origin in  $\bR ^{2N}$ satisfying
\begin{enumerate}

\item $Q_1,\ldots ,Q_N, \overline {Q_1},\ldots \overline{Q_N}$ span
  $\bC\otimes \bR ^{2N}$

\item 
$[Q_i, Q_j] \mbox{ is in the linear span of } \{ Q_1, \ldots , Q_N\}$ \quad
for $1 \leq i,j \leq N$.
 
\end{enumerate}
Then there exist complex coordinates $\{ z_1, \ldots , z_N\}$ on a
possibly smaller \nbd of the origin such that 
\[
 \{ Q_1, \ldots , Q_N\} = \{ \partial _{\overline{z_1}}, \ldots ,
 \partial _{\overline{z_N}}\} .
\]
\end{theorem}

To use this theorem to prove the lemma, let $L_1, \ldots , L_n$ be real
analytic and span the CR bundle $B$ near some point $p\in M$.  Choose
coordinates
\[
x_1, \ldots , x_{2n}, t_1, \ldots , t_k
\]
so that 
\[ 
\partial 
x_1, \ldots ,\partial  x_{2n}
\]
span $B \oplus \overline{B} $ at $p$. First write
\[
L_j= \sum_{p=1}^{2n} \alpha _{jp}(x,t)\partial _{x_k} +\sum_{m=1}^k
\beta _{jm}(x,t)\partial _{t_m}, \quad j=1, \ldots, n,
\]
with coefficients $\alpha$ and $\beta$ in $C^\omega$ 
and then extend each vector field to $\bR ^{2n+2k}$ by
\[
L_j= \sum_{p=1}^{2n} \alpha _{jp}(x,t+is)\partial _{x_k} +\sum_{m=1}^k
\beta _{jm}(x,t+is)\partial _{t_m}.
\]
Let  $\widetilde{B}$ be the linear span of these extended vector fields
together with the vector fields
\[
\partial _{t_m} +i \partial _{s_m}, \ \ \ m= 1, \ldots , k.
\]
Then, using the fact that the coefficients are holomorphic in $t+is$, 
$\widetilde{B}$ is involutive and so the above
Theorem applies:
\[
\widetilde{B}= \{  \partial _{\overline{z_1}}, \ldots ,
 \partial _{\overline{z_{n+k}} }\}.
\]
This gives us an embedding of a \nbd of $p$ in $M$ into $\bC ^{n+k}$
with
\[
B=( \ctm )\cap \tzeroone (\bC ^{n+k}).
\]
Thus the embedding is a generic CR embedding.

To find $\gamma _{ij}$ we assume that $\scriptO _i \cap \scriptO _j$ is
not empty.  Thus we have a diffeomorphism
\[
\Phi : f_i(\scriptO _i \cap \scriptO _j) \to f_j    (\scriptO _i \cap
\scriptO _j)
\]
with $\Phi _* B= B.$  This implies that each component of $\Phi$
is a $C^\omega$ CR function.  That is, for each local section $L$ of $B$
we have 
\[
L\Phi _j =0.
\]

Such a function on a generic submanifold of a complex manifold is the
restriction of a holomorphic function. See, for instance 
\cite[page 29, Corollary 1.7.13]{BER}.  We note for later use that the proof of this 
Corollary also shows that the extension is unique.

Let $\gamma _{ij}$ denote this extension.  If $\gamma _{ij}$ were
not a biholomorphism in some \nbd of $f_i(\scriptO _i \cap \scriptO
_j)$ then there would be a nonzero vector 
\[
v=\sum _
 {m=1}^{n+k} \alpha _{m }\partial z_{m}
\]
annihilated by $d\gamma _{ij}$ at some $p\in f_i(\scriptO _i \cap
\scriptO _j)$.  Since the submanifold is CR generic, $v=T - iJT$ for
some $T\in TM|_p$.  But since $\gamma _{ij}$
 has been extended
holomorphically 
\begin{eqnarray*}
d\gamma v &=& d\gamma T-id\gamma JT\\
&=& d\gamma T- iJ(d\gamma T).
\end{eqnarray*}
So if $\gamma _{ij} $ is not a biholomorphism, $d\gamma _{ij}$
annihilates a nonzero vector in $TM$ and this is impossible because 
$\gamma_{ij}$ is a real diffeomorphism on $f_i(\scriptO _i\cap \scriptO _j)$.
\end{proof}

We note that
\[
\gamma _{ik}=\gamma _{ij}\circ \gamma _{jk}
\]
on the domain $f_k(\scriptO _i\cap \scriptO _j\cap\scriptO _k)$.  To see this  let $p\in f_k(\scriptO _i\cap \scriptO _j\cap\scriptO _k)$.  Then there exists some $x\in \scriptO _i\cap \scriptO _j\cap\scriptO _k$ with $f_k(x)=p$.  Thus
\[
\gamma _{ik}(p)=\gamma _{ik}\circ f _k(x)=f_i(x).
\]
Further, 
\[
\gamma _{ij}\circ\gamma _{jk}(p)=\gamma _{ij}\gamma_{jk}\circ f_k(x)=\gamma _{ij}\circ f_j(x)=f_i(x)=\gamma _{ik}(p).
\]
By the uniqueness of extensions of holomorphic functions
 off of
generic submanifolds,  as
referenced in the previous proof, the fact that 
\[
\gamma _{ik}=\gamma_{ij}\circ\gamma_{jk}
\]
holds on $f_k(\scriptO _i\cap \scriptO _j\cap\scriptO _k)$ implies that it holds wherever it makes sense.


\section{Haefliger  structures}\label{Haefliger}

Following a well-known procedure in foliation theory, as in
\cite{Lawson} or \cite {Milnor} we generalize the objects in Lemma
\ref{CR} by dropping the requirements that the maps are smooth (indeed
analytic) and are embeddings.  As a reminder, we consider only paracompact manifolds.  
Let $A$ be some index set.

\begin{definition}
A Haefliger  CR structure of codimension $k$ on $M^{2n+k}$ consists of 

\begin{itemize}
\item 
An open covering $M=\bigcup _j \scriptO _j$,\ \ $j\in A,$
\item 
continuous maps $f_j:\scriptO _j \to \bC ^{n+k}$, 
\item  
local biholomorphisms $\gamma _{ij}$ of $\bC ^{n+k} $ defined for each pair $(i,j)$ such that $\scriptO _i\cap \scriptO _j\neq \emptyset$ satisfying:
\begin{enumerate}
\item 
$\gamma _{ik}=\gamma _{ij}\circ \gamma _{jk}$ at all points where both sides are defined, and

\item 
$f_i=\gamma _{ij}\circ f_j \text{ on } \scriptO _i\cap \scriptO_j$.
\end{enumerate}
\end{itemize}
\end{definition}
  
Lemma \ref{CR} shows that a $C^\omega$ CR structure admits a Haefliger  CR structure of a special kind.  Namely, each $f_j$ is an embedding.  

Let $A_{ij}$ and $B_{ij}$ denote, respectively, the domain and codomain of $\gamma _{ij}$.   So
 $\gamma _{ij}$ induces a bijection $\gamma _{ij*}: T^{1,0}(A_{ij})\to T^{1,0}(B_{ij})$.  We may restrict $\gamma _{ij*}$ to the image of $\scriptO _i\cap \scriptO _j$ and then pull back to $ M$ to obtain transition functions $g_{ij}$ that patch $\scriptO _i\times \cnk$ to $\scriptO _j\times \cnk $ over $\scriptO _i\cap \scriptO_j$, and so determine a vector bundle:

\begin{definition}
The normal bundle $\nu$ of a Haefliger  CR structure is the complex $(n+k)$-dimensional 
vector bundle over $M$ with transition functions $g_{ij}$.
\end{definition}

\begin{theorem}\label{easy}
If $B$ is a $C^\omega$ CR structure then $(\ctm )/B$ is isomorphic to the normal bundle of a Haefliger  CR structure.
\end{theorem}
\begin{proof}
By Lemma \ref{CR}, $M$ admits a Haefliger  CR structure.  Each $f_i$ is a
generic CR map, so we have over $\scriptO _i$, 
\[
B\cong f{_i}_*B=(\ctm )\cap \tzeroone .
\]
We claim this implies  
\begin{equation}\label {quotient}
\left( (\ctm )/B \right) \vert _{\scriptO_i} \cong f_i^*T^{1,0} .
\end{equation}
To see this consider the map $\ctm \to \tonezero$ given by, again over
$\scriptO _i$,
\[
\Phi\ :\ Y \mapsto f{_i}_*Y -iJf{_i}_*Y
\]
where $J$ is the standard anti-involution on $\bC ^{n+k}$.  Note that $\Phi (Y)=0$
\iff 
\[
Y\in (\ctm )\cap \tzeroone = f{_i}_*B .
\]
Thus the kernel of $\Phi$ has complex dimension $n$, which implies that
$\Phi$ is surjective.  Then  \eqref{quotient} follows.

But the transition functions for $T^{1,0}|_{f_i(M)}$ are just the derivative maps $g_{ij}$.  Thus  $(\ctm )/B$ is isomorphic to the normal bundle of this Haefliger  CR structure.
\end{proof}
The same conclusion holds if $B$ is homotopic (even through $2n$-plane
bundles which may not be almost CR) to a $C^\omega$ CR structure.

We can now state the first of two basic steps in the proof of our main
theorem.  Here there is no restriction on the manifold.  But for the
second step, Theorem \ref{lifttheorem}, the topology of the manifold is
restricted.  In particular, $M$ needs to be open.

\begin{theorem}\label{CRn}
Let $B$ be a smooth  almost  CR structure of codimension $k$ on $M^{2n+k}$.  If $(\ctm )/B$ is isomorphic to the normal bundle of a Haefliger  CR structure then $B$ is homotopic through almost CR structures of codimension $k$ to a  $C^\omega$ CR structure.
\end{theorem}
The next two sections contain the proof of this theorem.
\begin{corollary}
Let $B$ be a smooth CR structure.  Then $B$ is homotopic to a $C^\omega$ CR structure if and only if 
$(\ctm )/B$ is isomorphic to the normal bundle of a Haefliger  CR structure.
\end{corollary}
It is not known if every smooth CR structure is homotopic to a
$C^\omega$ CR structure.  It is possible that this is true on some
manifolds and false on others.  (It is certainly true locally.)  We
shall see in Theorem \ref{main} that it is true provided $M$ satisfies some topological
restrictions.

\section{The CR embedding}
In this section we reduce the proof of Theorem  \ref{CRn} to an
h-principle argument which is then provided in the next section.

\begin{lemma} Given a Haefliger CR structure on $M^{2n+k}$ there
exists a manifold $X$ of real dimension $4n+3k$ and an embedding $\iota
:M\to X$ such that
\begin{enumerate}
\item 
$X$ is a fiber bundle over $M$ with complex structure on the
fibers.\label{X1}
\item 
$X$ admits a foliation $\scriptF ^{2n+k}$ transverse to the fibers.\label{X2}
\item 
The normal bundle $\nu$ of the Haefliger structure is isomorphic to\label{X3}
\[
\iota^*T(X)/T(M).
\]
\end{enumerate}

\end{lemma}

\begin{proof}
Let $M$ have local charts\[
\phi _i:\scriptO _i\to \bR ^{2n+k}.
\]
Set $\Omega$ equal to the disjoint union 
\[
\Omega = \bigsqcup _i \, (\phi _i(\scriptO _i)\times \bC ^{n+k}),
\]
and define an equivalence relation on $\Omega$ by setting
\[
(x_i,y_i)\sim (x_j,y_j)
\]
whenever
\begin{eqnarray}
x_j &=& \phi _j\phi _i^{-1} (x_i) \quad \mathrm{and}\\
y_j &=& \gamma _{ji}(y_i).
\end{eqnarray}
Here $\gamma _{ji}$ are the local biholomorphisms of the Haefliger
structure.  
Denote the quotient space, with the quotient topology, by $Z$. Note
that the projection map
\[  \Omega \to Z
\]
is an open map and that the projection map
\begin{equation}\label{pi}
\pi\ :\ Z\to M
\end{equation}
is well-defined.   The
multi-valued map of $M$ into $\Omega$
\[
p\mapsto \bigsqcup _i \,(p,f_i(p))
\]
for $p\in \scriptO _i$ passes down to a well-defined map 
\[
\iota :M\to Z.
\]
Clearly, $\pi\circ\iota$ is the identity.

We shall see that $Z$ is Hausdorff near $\iota (M)$. This provides 
the manifold $X$.  To show that $Z$ is Hausdorff near the image of $M$
let $a=[(x_i,y_i)] $ and $b=[(x_j,y_j)]$ be points of $Z$.  If
$x_j\neq \ \phi _j\phi _i^{-1}(x_i)$, then we can separate $a$ and $b$ by disjoint open sets.  So now
let $a=[(x_i,y_i)]$ and $b=[( \phi _j\phi _i^{-1}(x_i), y_j)]$.  We may assume
(by staying close enough to $\iota (M)$) that $y_i$ is in the domain of
$\gamma _{ji}$.  Thus if $y_j\neq \gamma_{ji} (y_i)$ we may again
separate $a$ and $b$ while if  $y_j= \gamma_{ji} (y_i)$ then $a=b$.

Note that the transition functions for $X$ are
\begin{equation} \label{X}
\psi _{ji} (x,y)=(\phi _j\phi _i^{-1}(x), \gamma_{ji}(y))
\end{equation}
for $x\in \phi _i(\scriptO _i\cap \scriptO_j)$ and $y$ close enough to $\gamma _{ji}(f_i)$.  This
implies that the two natural foliations on $\Omega$ corresponding to
each factor in $\phi _i(\scriptO_i)\times \bC^{n+k}$ pass down to $X$.  In
particular, the complex structure on each fiber $\pi ^{-1}(p)$ is
well-defined. (Here the map $\pi$ of \eqref{pi}is restricted to
$X\subset Z$.)
Thus
\eqref{X1} and \eqref{X2} of the Lemma hold. Further, the transition functions for $T(X)$ are
\[
\psi_{ji*}(x,y)=((\phi _j\phi_i^{-1})_*(x),\gamma_{ji*}(y)).
\]
Restricting to $\iota (M)$ we have
\bea
\psi_{ji*}(x)&= ((\phi _j\phi_i^{-1})_*(x),\gamma_{ji*}(f_i(x)))\\
&= ((\phi _J\phi_i^{-1})_*,g_{ji}(f_i(x))).
\eea
So the transition functions for the normal bundle of $\iota (M)\subset
X$ are the same as those for $\nu$, the normal bundle of the Haefliger
structure.  This establishes \eqref{X3}.
\end{proof}

To  explain what comes next, we  review a well-known
procedure in foliation theory.  Let $\scriptM$ have a foliation
$\scriptF$ of codimension $q$.  If 
$\Phi : M\to \scriptM$ 
is transverse to the leaves of $\scriptF$, in the sense that $\Phi_*T(M)$
and $T(\scriptF)$ together span $T(\scriptM)$ at each point of the image
of $M$, then, as is easily seen, $\scriptF$ induces a foliation on $M$, of the same 
codimension.  By the Gromov-Phillips Theorem (see, for instance, \cite[Section 4, pp.~293--297]{Milnor} 
when $M$ is open, it is enough to have a bundle map
\[
\Psi : T(M) \to T(\scriptM)
\]
satisfying the tranversality condition, for then the base map of
$\Psi$ may be deformed to a transverse map $\Phi$.  

We rephrase this as: An open manifold $M$ admits a codimension-$q$
foliation provided there exists a vector bundle map $\Psi$ such that the
composition
\[
T(M) \stackrel{\Psi}{\longrightarrow}T(\scriptM)\longrightarrow 
T(\scriptM) /T(\scriptF) 
\]
is surjective, where $\scriptF$ is a codimension-$q$ foliation of $\scriptM$.  

To state the modification of the Gromov-Phillips Theorem we use, we
first need some definitions.   Let $\pi : X \to M$ be the projection provided by the Lemma.   
Recall that each fiber $\pi ^{-1}(p)$ is a complex manifold.  Set

\[
T_f =\bigcup _{p\in M}T (\pi ^{-1}(p))\subset T(X)
\]
and
\[
\tonezero _f =\bigcup _{p\in M}\tonezero (\pi ^{-1}(p))\subset \ctx.
\]
The subscript is to remind us that we are looking at tangents to the
fiber. The form of the transition functions for $X$, see \eqref{X}, 
shows that the map
\[
\mu \ :\ \ctx\to\Tonezero _f
\]
is well-defined.

Although the following is patterned upon the Gromov-Phillips Theorem,
the restriction to open manifolds is not necessary here. 
\begin{theorem}\label{gromov}
If there exists a surjective bundle map 
\[
\Psi: \ctm \to \tonezero _f\vert _M 
\]
then there exists a smooth map
\[
F:M\to X
\]
such that

\begin{enumerate}
\item The composition 
\[
\mu F_* :\ctm \to \ctx \to \Tonezero _f
\]
is surjective.
\item
$\mu _*F_*$ is injective when restricted to $T(M)$.
\item
$B_1= \ker \mu F_*$ is homotopic to $B= \ker \Psi $.
\end{enumerate}
\end{theorem}
This  is the heart of the proof of Theorem \ref{CRn}. We prove
it in the next section by means of Gromov's h-principle.  Here let us
first see how it implies the conclusion of Theorem \ref{CRn} and then how we use the hypothesis
\[
(\ctm )/ B \cong \nu
\]
to obtain the surjective map
\[
\Psi: \ctm \to \tonezero _f\vert _M .
\]
We first show that $B_1$ defines a CR structure.

\begin{lemma}\label{involutive}
The bundle $B_1$ is involutive.  
\end{lemma}

\begin{proof}
 It suffices to assume that $a$ and $b$ are local sections of $B_1$ and to  show that
\[
\mu F_*[a,b]=0.
\]
We work in local coordinates.  So we write 
\[
F(x)=(f(x),y(x)) \in \phi _i(\scriptO)\times\bC ^{n+k}.
\]
For any $a\in \ctm$, we have (note that we now use $a(y)$ to mean the action of the vector $a$ on the function $y$ and we employ the summation convention)
\[
F_*(a) = a(f_j)\partial _{x_j}+a(y_k)\partial
_{z_k}+a(\overline{y_k})\partial _{\overline{z_k}}
\]
and 
\[
\mu F_*=a(y_k)\partial _{z_k} .
\]

It follows that
\[
B_1= \ker \mu  F_* =\{a\in \ctm : a(y_k)=0 \text{ for }k=1, \ldots
, n+1\}
\]
and this is involutive.
\end{proof}

\begin{lemma}
\[
B_1\cap \overline {B_1} = \{ 0\}.
\]
\end{lemma}

\begin{proof}
This is a consequence of the second condition of the Theorem.
\end{proof}

Thus $B_1$ is  a CR structure homotopic to $B$.  (The proof of Theorem
\ref{gromov} will show that this homotopy is through almost CR
structures.)

Finally, we need the map $\Psi$.

\begin{lemma}\label{f}
If $(\ctm )/B$ is isomorphic to the Haefliger normal bundle $\nu$ then 
there exists a surjective bundle map 
\[
\Psi: \ctm \to \tonezero _f|_M
\]
with $B= \ker \Psi$.
\end{lemma}

\begin{proof}
We have
\[
\ctm \to ( \ctm )/B  \stackrel{\alpha}{ \to}\nu \stackrel {\beta}{\to} \tonezero _f|_M.
\]
Since $\alpha$ and $\beta$ are isomorphisms,  the composite map $\Psi$
is surjective and has kernel $B$.
\end{proof}


\section{An application of the h-principle}
The general set-up is to define a subset $\scriptR \subset \Jone$,
called a relation, with the property that $j^1(F)\in \scriptR$ implies
that $ F$ satisfies the conclusions of Theorem \ref{gromov}.  See
\cite{EM}, \cite{Gr}, \cite{Ja} for discussions of the h-principle.

Again choosing local coordinates $x\in \bR ^{2n+k}$ for $M$, $y\in \bR ^{2n+k}$, and  $z\in \bC ^{n+k}$ for $\nu$ we see that 
\[
\mu F_* : \ctm \to \tonezero _f
\]
is surjective if the rank of the matrix
\[
\left(\frac {\partial z_i}{\partial x_j}\right),\quad\quad 1\leq i\leq n+k \text{ and } 1\leq j\leq 2n+k
\]
is maximal (i.e., equal to $n+k$) and that
\[
\mu  F_* : T(M) \to \tonezero _f
\]
is injective if the rank of the $(2n+2k)\times (2n+k)$ matrix 
\[
\left(\begin{array}{c}
\frac {\partial z_i}{\partial x_j}\\
\\
\frac {\partial \overline{z_i}}{\partial x_j}
\end{array}
\right)
\]
is also maximal.

So we define our relation by
\[
\scriptR = \{ (p,c,a^j_i): \rank (a) =n+k,\quad \rank \left(\begin{array}{c}
a\\
\overline{a}
\end{array}
\right) = 2n+k\}
\]
where, as above
\[ 1\leq i\leq n+k,\ \ 1\leq j\leq 2n+k.
\]
We claim that $\scriptR$ is an open and ample subset of $\Jone$.  For
the definition of \textit{ample} see any of the above cited works or \cite{JL};
for the proof of the claim see \cite {JL},  pages 157--159.  It is a consequence of the h-principle that the desired $F$ exists provided we can find a section
\[
\sigma\ :\ M\to \scriptR .
\]
That is, the existence of a formal solution implies the existence of a genuine solution. 

To see that we do have a formal solution, we identify a bundle map
\[
\sigma\ :\ \ctm \to \tonezero _f
\]
with a section
\[
\tilde{\sigma }\ :\ M\to \jone
\]
by writing
\[
\sigma (\partial _{x_j}\vert _p)=h^j_i\partial _{z_i}\vert _{\sigma (p)}
\]
and then setting
\[
\tilde {\sigma }(p)=(p,\sigma (p),h^j_i).
\]
As long as $\sigma$ is surjective and its kernel satisfies $B\cap
\overline{B}=\{ 0\}$, so 
\[
\sigma :\ T(M)\to T_f^{1,0}
\]
is injective, 
 $\tilde{\sigma }$ is a formal solution,
\[
\tilde{\sigma }\ :\ M\to \scriptR \subset \jone .
\]
In this context, Lemma \ref{f} asserts the existence of a formal solution.

The h-principle asserts more than the existence of a genuine solution.  In fact, the genuine solution is obtained by a deformation
\[
\sigma _t\ :\ M\to \scriptR \ \ \ 0\leq t\leq 1
\] 
such that $\sigma _0=\tilde{\sigma }$ and $\sigma _1=j^1(F)$.
Now set 
\[
B_t=\text{ ker } \mu \sigma _t.
\]
We have $\rank B_t=n$ and $B_t\cap \overline{B_t}=\{ 0\}$ for all $t$, and moreover $B_1$ is involutive.

Thus $B_t$ represents a deformation of $B$ through almost CR
structures to a true   CR structure.

Finally, once we have $F:M\to X  $ we may deform $F$, maintaining the open conditions on the ranks of the matrices, to a $C^\omega$ map.

The proof of Theorem \ref{CRn} is complete.


\section{Classifying spaces}

We first note that the definition of a Haefliger CR structure makes
sense (for fixed $n$ and $k$) if $M$ is merely a topological space.
To specify $n$ and $k$  we  speak of an $(n,k)$ \haef.  (Actually, only
$n+k$ appears in the definition, so we occasionally speak of an $(n+k)$ Haefliger structure.)  
Also if $M$ and $N$ are topological
spaces and $N$ has an $(n,k)$ \haef \ then any 
 continuous map $\phi:M\to N$ induces an $(n,k)$ Haefliger  CR
 structure on $M$.
 
\begin{definition}
 Two $(n,k)$ Haefliger  CR structures on $N$ are homotopic if
 there exists an   $(n,k)$ Haefliger  CR structure on $N\times I$
 which pulls back to the given structures on $N\times \{0\}$ and
 $N\times \{1\}$.
\end{definition}

\begin{definition}
A topological space $\scriptB_{n,k}$ with an $(n,k)$ Haefliger  CR structure  is a classifying space for\ $(n,k)$  Haefliger  CR structures if 
\begin{itemize}
\item Every $(n,k)$ Haefliger  CR structure  on a manifold $N$ is induced by a continuous map $\phi :N\to \scriptB_{n,k}$.
\item Two $(n,k)$ Haefliger  CR structures on $N$ are homotopic  if and only if the corresponding continuous maps into $\scriptB_{n,k}$ are homotopic.
\end{itemize}
The associated $(n+k)$-dimensional normal bundle over $\scriptB_{n,k}$ is denoted by $\nu_{n,k}$.
\end{definition}

\begin{theorem}
There exists a classifying space for $(n,k)$ Haefliger  CR structures.
\end{theorem}

This result is a particular case  of  the existence of classifying
spaces for topological groupoid structures (as proved by Haefliger \cite{Haef1}, \cite{Haef2} and
others---see \cite[page 48]{La2} or \cite[page 312]{Milnor} for
expository accounts).  Indeed, 
these are precisely the classifying spaces  for complex structures studied by
Landweber \cite{Land} and by Adachi \cite{Ad} and usually denoted by $B\Gamma_q^{\bC}$.  In the present case, the codimension $q$ is taken to be $n+k$.  
Here $\Gamma_q^{\bC}$ denotes the topological groupoid of germs of biholomorphisms between open subsets of 
$\bC ^q$.


\section{The lifting problem}\label{lift}

We have started with a complex subbundle $B\subset \ctm$ and have set $\nu= (\ctm)
/B$.  Using a Hermitian metric on $\ctm$ we may also consider $\nu$ as a
subbundle of $\ctm$ complementary to $B$ in $\ctm$.  Our aim is to show that the subbundle $\nu
\subset \ctm$ of rank $n+k$ is the normal bundle of some $(n+k)$
Haefliger structure on $M$.  Since bundles over $M$ are classified by homotopy
classes of maps of $M$ into $BGL(n+k)$ (we denote by $GL(m)$ the complex general linear group 
$GL(m,\mathbf C)$ ) and Haefliger structures are
classified by homotopy classes of maps of $M$ into $\scriptB_{n,k}$, the
following set-up is natural.
For any complex vector bundle $\xi$ of rank $m$ over a paracompact space $X$, let
$\mathrm{cl}\,(\xi): X \to BGL(m)$ denote a classifying map for $\xi$.  
Thus we have the commutative
diagram

\[
\renewcommand{\arraystretch}{2.0}
\begin{array}{ccc}
M&\xrightarrow{\scriptstyle{\mathrm{cl}\,(\nu)\times \mathrm{cl}\,(B) }}&BGL(n+k)\times BGL(n)\\
&\textrm{cl}\,(\ctm)
\searrow &\Bigg\downarrow\rlap{$B(\mathrm{i})$}\\
& & BGL(2n+k)  
\end{array}
\]
where i denotes the inclusion of $GL(n+k) \times GL(n)$ into $GL(2n+k)$ by means of block matrices and 
$B(\textrm{i})$ is the map of classifying spaces induced by this inclusion.  

As above, let
$\nu _{n,k}$ denote the normal bundle on the classifying space
$\scriptB_{n,k}$.  Note that $\nu _{n,k}$ defines a map $\mathrm{cl}\,(\nu_{n,k}) : \scriptB_{n,k} \to BGL(n+k)$,
unique up to homotopy.

\begin{theorem}\label{lifttheorem}
If
\begin{equation}
H_{p}(M^{2n+k};\, \mathbf Z)=0 \,\text{ for }\, p\geq n+k+1\label{eqH=0}
\end{equation}
then  there exists a lift
\[\renewcommand{\arraystretch}{2.0}
\begin{array}{ccc}
& &\scriptB_{n,k} \times BGL(n)\\
&
\nearrow&\Bigg\downarrow\rlap{$\mathrm{cl}\,(\nu_{n,k}) \times  B(\mathrm{id)}$}\\
M&\xrightarrow{\scriptstyle{\mathrm{cl}\,(\nu)\times \mathrm{cl}\,(B)}}&BGL(n+k)\times BGL(n).\\
\end{array}\]
\end{theorem}

\begin{proof}
The homotopy fiber $\scriptF = \scriptF_{n,k}$ of the map
$\mathrm{cl}\,(\nu_{n,k}) : \scriptB_{n,k} \to BGL(n+k)$
is ($n+k$)-connected (cf. \cite{Ad} and \cite{Land}, although the latter paper needs to be supplemented by Gromov's h-principle for totally real immersions, as given in \cite{GR} and \cite{Gr}).
Since the obstructions to lifting a map
\begin{equation}\label{map}      
M\to BGL(n+k)\times BGL(n)
\end{equation}
lie in $H^{j+1}(M;\,\{\pi _j(\scriptF)\})$ for
the local coefficient system $\{ \pi_j(\scriptF )\}$, and since $\pi
_j(\scriptF )=0$ for $ 0\leq j\leq n+k$, 
the nonzero obstructions can
lie only in $H^{j+1}(M;\,\{ \pi _j(\scriptF)\})$ for $j\geq n+k+1$.
(For discussion of local coefficient systems, also known as bundles of coefficients or systems of local groups, we refer to Steenrod's book \cite{St51} on fiber bundles, or to his paper \cite{St43}.  For a more recent account dealing with fibrations, see G.~Whitehead's book \cite[Chapter VI]{Whitehead}.) 
Since  complex general linear groups are connected, their classifying spaces are simply connected, 
and so the map 
$\mathrm{cl}\,(\nu_{n,k}) : \scriptB_{n,k} \to BGL(n+k)$, which we
view as a fibration, has simple local coefficient systems in the sense 
of \cite{St43}.  This means that there is a unique isomorphism between
the homotopy groups of any two fibers.  
It follows that the local coefficient systems over $M$ appearing here are also simple.
As a result, the cohomology groups here are ordinary cohomology groups
and we can apply the usual universal coefficient theorem which
expresses cohomology in terms of integral homology (note that 
homotopy  groups beyond the first one are abelian and so the universal
coefficient theorem applies).
Hence if \eqref{eqH=0} holds then by the universal coefficient theorem
we also have
\[
 H^{j+1}(M;\, \{\pi _j(\scriptF)\}) = 0 \text{ for } j\geq n+k+1
\]
and so every map \eqref{map} lifts to a map 
\[
M\to \scriptB_{n,k} \times BGL(n).
\]
\end{proof}

\begin{remark}
The hypothesis of Theorem \ref{lifttheorem} can be slightly weakened,
Namely, it suffices to require that the integral homology groups in
(\ref{eqH=0}) vanish for $p \geq n+k+2$, and be free abelian for $p =
n+k+1$.  In particular, this weakened hypothesis is satisfied if the
manifold $M^{2n+k}$ has the homotopy type of a CW complex of dimension
at most $n+k+1$.  This last condition is satisfied if $M^{2n+k}$ has
the homotopy type of an open $(n+k+2)$-manifold, which follows from
\cite[Lemma 1]{Ph}.  In short, if $M^{2n+k}$ is homotopic to an open
$(n+k+2)$ manifold, then the lift exists and Theorem \ref{mainresult} holds in
this slightly more general case.

\end{remark}

\section{The theorem}

\begin{theorem}\label{main}
If $M^{2n+k}$ satisfies the condition \eqref{eqH=0} of Theorem \ref{lifttheorem} then
every smooth   almost CR structure of codimension $k$ on $M$ is
homotopic to a $C^\omega$ CR structure of codimension $k$.  
In particular, every $C^\infty$ CR structure may be deformed to a
$C^\omega$ CR structure.
\end{theorem}

\begin{proof}
The lifting $M\to \scriptB_{n,k} \times BGL(n)$ of Theorem \ref{lifttheorem}
shows that $\nu \oplus B$ is isomorphic to some bundle 
$\nu^0 \oplus B^0$ with $\nu ^0$ the normal bundle of some $(n,k)$ Haefliger structure
on $M$ and $B^0$ a complex bundle of rank $n$.  The map 
$\mathrm{cl}\,(\nu) \times \textrm{cl}\,(B)$ is homotopic to the map $\mathrm{cl}\,(\ctm)$ and so, using
the fact that the diagram is commutative we see that 
\[
\nu ^0\oplus B^0\cong \ctm .
\]
In particular, $\nu ^0$ and $B^0$ can be identified with subbundles of
$\ctm $.  Further, since the map $\mathrm{cl}\,(\nu_{n,k})\times B(\textrm{id})$ 
is the identity on the $BGL(n)$ factor, we have that $\nu$
is isomorphic to $\nu ^0$. 
Thus Theorem
\ref{CRn} may be applied to show that  $B$ may be deformed to some CR structure.  Further, since
this CR structure is induced by a map $F:M\to X$, as in Theorem
\ref{gromov}, the map $F$ itself may be deformed to yield a $C^\omega$
CR structure.  This last step just uses that the real analytic
functions are dense in the set of smooth functions.
\end{proof}

Note that the requirement that the integral homology of $M$ vanish in
high dimensions excludes compact manifolds, provided that $n>1$.  In
the case $n=1$ the theorem holds for a trivial reason:  Each almost CR
structure with $n=1$ is automatically involutive.

\begin{remark}
This theorem also holds for $k=0$, as in \cite{Ad} and
\cite{Land}, but our proof does not.  The reason for this, in a very
similar case, is explained in \cite{JL}.  
\end{remark}

In view of the analogy between foliations and involutive structures it is
natural to try to place into this new context Bott's necessary
condition for a subbundle of $TM$ to be homotopic to a foliation.  See
\cite{Bott68}, \cite{BottICM}, or \cite{BottLectures} for Bott's original argument 
and some of its further consequences.  

\end{document}